\documentclass[11pt]{article}
\usepackage{amsthm}
\usepackage{amssymb}
\usepackage{amsmath,enumerate}
\usepackage{comment}
\usepackage{thm-restate}
\usepackage{url} 
\usepackage{hyperref}
\usepackage{graphicx}
\usepackage{subfigure}
\usepackage[noabbrev,capitalise]{cleveref}
\usepackage[affil-it]{authblk}
\usepackage{color}
\usepackage[normalem]{ulem}

\usepackage[margin=1in]{geometry}

\theoremstyle{plain}
\newtheorem{thm}{Theorem}[section]
\newtheorem{lem}[thm]{Lemma}

\newtheorem{cor}[thm]{Corollary}

{\noindent \emph{Proof.} {}{#1}{}}{\hfill
	$\Diamond$\vspace{1em}}

\theoremstyle{plain} 
\newcommand{\thistheoremname}{}
\newtheorem{genericthm}[section]{\thistheoremname}

\theoremstyle{definition}

\def\less{\setminus}

\def\dfn#1{{\sl #1}}

\def\less{\setminus}

\title{Multicolor bipartite Ramsey  number of double stars}
\author{Gregory DeCamillis\thanks{Supported by  the Honors Undergraduate Thesis (HUT) Scholarship  and the Allyn M. Stearman Research Scholarship for undergraduate students at the University of Central Florida.   E-mail address: {\tt  Gregory.Decamillis@ucf.edu}.}\,\, and  Zi-Xia Song\thanks{Supported by  NSF grant    DMS-2153945. E-mail address: {\tt Zixia.Song@ucf.edu}.}}
 
  \affil{ 
  { \small {Department  of Mathematics, University of Central Florida, Orlando, FL 32816, USA}}  
     }

\date{}

\begin{document}
\maketitle
\begin{abstract}
For positive integers   $n, m$, the double star  $S(n,m)$ is  the graph consisting of the disjoint union of two stars  $K_{1,n}$  and  $K_{1,m}$ together with an edge  joining their centers.  Finding  monochromatic copies  of double stars in edge-colored complete bipartite graphs has attracted much attention.  The $k$-color bipartite Ramsey number of $ S(n,m)$, denoted by $r_{bip}(S(n,m);k)$,  is the smallest  integer $N$  such that,  in any $k$-coloring of the edges of the complete bipartite graph $K_{N,N}$,  there is a monochromatic copy of $S(n,m)$. The study of bipartite Ramsey numbers was initiated in the early 1970s by Faudree and Schelp and, independently, by Gy\'arf\'as and Lehel.   The exact value of  $r_{bip}(S(n,m);k)$ is only known  when  $n=m=1$. Applying the  Tur\'an argument in the bipartite setting, here we prove that   if $k=2$ and $n\ge m$,  or $k\ge3$ and $n\ge 2m$, then 
\[ r_{bip}(S(n,m);k)=kn+1.\]

\end{abstract}
 
\baselineskip 16pt

\section{Introduction}
All graphs   considered in this paper  are finite, simple, and undirected.   Give a graph $G$, we use $e(G)$ to denote the number of edges in $G$.  We say $G$ is \dfn{$H$-free} if $G$ does not contain a given graph $H$ as a subgraph. For  two disjoint sets $A, B\subseteq V(G)$,   we  say that       $A$ is \dfn{anti-complete}     to $B$   if no vertex in $A$ is adjacent to any vertex in $B$;  we use $e_G(A, B)$ to denote the number of edges between $A$ and $B$ in $G$.  Finally we use $d_G(v)$ to denote the degree of a vertex $v$ in $G$. 
For any positive integer $k$, we write  $[k]$ for the set $\{1,2, \ldots, k\}$. A \dfn{$k$-edge-coloring} of a graph $G$ is a function $\tau:E(G)\to [k]$.   We say that $\tau$ is \dfn{proper} if $\tau(e)\ne \tau(e')$  for any two adjacent edges $e, e'$   of $G$. We think of the set  $[k]$ as a set of colors, and we may identify a member of $[k]$ as a color, say, color $k$ is blue.   Given a graph $H$,  we write \dfn{$G \rightarrow (H;k)$} if    every  $k$-edge-coloring of   $G$ contains a monochromatic  ${H} $. Then $G \nrightarrow (H;k)$ if there exists $\sigma:E(G)\to [k]$ such that $G$   contains no monochromatic copy of $H$ under $\sigma$; such a $k$-edge-coloring $\sigma$ is called a \dfn{critical $k$-coloring} of $G$.   The \dfn{$k$-color Ramsey number} $r(H; k)$ of    $H$   is the smallest integer $n$ such that   $K_n \rightarrow (H;k)$. Computing $r(H; k)$ in general is a notoriously difficult problem.

  There have been  many generalizations of Ramsey numbers  over the years; one natural generalization is  replacing the underlying complete graph by a complete bipartite graph.  The \dfn{$k$-color bipartite Ramsey number} of a graph $H$, denoted by $r_{bip}(H;k)$,  is the smallest  integer $N$  such that,  in any $k$-edge-coloring of  of the complete bipartite graph $K_{N,N}$,  there is a monochromatic copy of $H$. Faudree and Schelp~\cite{FSpath} and,  independently,  Gy\'arf\'as and Lehel~\cite{GyarfBipPath}  initiated the   study of bipartite Ramsey numbers   in the early 70s;      they  both  determined the exact value of $r_{bip}(P_n;2)$. The natural extension to $k$-color bipartite Ramsey number of  paths and cycles has been considered recently.  Buci\'c, Letzter and Sudakov~\cite{Bucicr3, Bucicr3P} determined asymptotically the $k$-color bipartite Ramsey number of paths and even cycles for $k\in\{3,4\}$ and     bounds for  the $k$-color bipartite Ramsey number of paths and cycles for $k\ge5$.  
 Yan and Peng~\cite{2-cycles}     determined the exact value of  the bipartite Ramsey number of $r_{bip}(C_{2n}, C_{2m})$ for all $n\ge m\ge5$.    Beineke and Schwenk~\cite{Beineke} first studied the bipartite Ramsey number of $K_{s,s}$. 
 The known bounds for $r_{bip}(K_{s,s};2)$~\cite{Conlon, hattingh, Irving} are,  
 \[\left(1+o(1)\right)\frac{s}e({\sqrt2})^{s+1} \le r_{bip}(K_{s,s};2)\le \left(1+o(1)\right)2^{s+1}\log s.\]
  The upper bound, due to Conlon~\cite{Conlon},   improves upon the previous known upper bound of Irving~\cite{Irving} from 1978,     and its proof uses  the   Tur\'an argument   in the bipartite setting as used in~\cite{Irving}.  
 
 In this paper we  study  the $k$-color bipartite Ramsey number of double stars, where for positive integers   $n, m$,
the \dfn{double star} $S(n,m)$     is the graph consisting of the disjoint union of two stars  $K_{1,n}$  and  $K_{1,m}$ together with an edge  joining their centers.   Note that $S(1,1)=P_4$.   Finding monochromatic copies of double stars in edge-colored complete bipartite graphs has been investigated by Gy\'arf\'as~\cite{GY1}, Mubayi~\cite{Mubayi2001},  Liu, Morris and  Prince~\cite{Liu2008}, DeBiasio, Krueger and S\'ark\"ozy~\cite{DKS},  DeBiasio,  Gy\'arf\'as, Krueger, Ruszink\'o and S\'ark\"ozy~\cite{DeBiaBip}, and S\'ark\"ozy~\cite{Sarkozy2022}.  
 It is worth noting that   the exact value of the $2$-color Ramsey number $r(S(n,m);2)$ is not completely known yet;    little is known regarding the $k$-color Ramsey number $r(S(n,m);k)$ except that $r(P_4;k)$  is known for all $k\ge2$.   Very recently, Ruotolo and the second author~\cite{doublestar} proved that  if  $k$ is odd and $n$ is sufficiently large compared with  $m$ and  $k$,  then 
\[ r(S(n,m);k)=kn+m+2.\]
 It turns out  that  $r_{bip}(S(n,m);k)$   behaves more nicely than  $r(S(n,m);k)$. Recently, DeBiasio, Gy\'arf\'as, Krueger, Ruszink\'o, and S\'ark\"ozy~\cite{DeBiaBip} determined the exact value of  $r_{bip}(P_4;k)$. 
\begin{thm}[DeBiasio, Gy\'arf\'as, Krueger, Ruszink\'o and S\'ark\"ozy~\cite{DeBiaBip}]\label{t:P4}
\[r_{bip}(P_4;k)= \begin{cases}
k+1 &\text{ if } k\in \{1,2,3\}\\
6&\text{ if } k=4\\
2k-3   &\text{ if } k\ge5.
\end{cases}
\]
\end{thm}

Here we prove the following main result.   

\begin{restatable}{thm}{doublestar}\label{t:main}
 If  $k=2$ and $n\ge m$, or $k\ge3$ and $n\ge2m$, 
then
 \[r_{bip}(S(n,m);k)=  kn+1.\]
\end{restatable}
 
Following  Irving~\cite{Irving} and  Conlon~\cite{Conlon}, our proof utilizes the Tur\'an argument in the bipartite graph setting. It is worth noting that the method we developed here does not seem to work for the case when $n=m$, as  $r_{bip}(S(n,n);k)$ behaves quite differently. 
Using  the $k$-color bipartite Ramsey number of $P_4$,  the construction in~\cite[Proposition 1.7]{DeBiaBip} shows  that    $K_{p,p}\nrightarrow (S(n,n);k)$, where $p=(r_{bip}(P_4;k)-1)n$.  Combining this with \cref{t:P4}  
leads to    \cref{c:lower}.   

  \begin{cor}[DeBiasio, Gy\'arf\'as, Krueger, Ruszink\'o and S\'ark\"ozy~\cite{DeBiaBip}]\label{c:lower}
\[r_{bip}(S(n,n);k)\ge \begin{cases}
3n+1 &\text{ if } k=3\\
5n+1&\text{ if } k=4\\
(2k-4)n+1   &\text{ if } k\ge5.
\end{cases}
\]
\end{cor}
 
     From the proof of \cref{t:main} when $k\ge3$, one can easily deduce the following upper bound for 
 $r_{bip}(S(n,n);k)$. We provide a proof here for completeness. 

\begin{cor}\label{c:upper1} For all $k \ge 3$ and $n\ge 1$, we have 
\[r_{bip}(S(n,n);k) \le \left\lfloor\left(1+\sqrt{1-\frac{2}{k}}\right)kn\right\rfloor + 1.\]
\end{cor}

\begin{proof}
Let $p: = \left\lfloor\left(1+\sqrt{1-\frac{2}{k}}\right)kn\right\rfloor + 1$ and $G:=K_{p,p} $.     Let $\tau: E(G)\rightarrow [k]$ be a  $k$-edge-coloring of $G$ with color classes $E_1, \ldots, E_k$.   For each $\ell\in[k]$,  let $G_\ell$ be the spanning subgraph of $G$ with edge set $E_\ell$.   Suppose  $G_\ell$ is $S(n,n)$-free for each $\ell\in[k]$. Since $p\ge3n+1$, we see that $\max\{np, 2n(p-n)\}=2n(p-n)$. By \cref{l:bipartite}, $e(G_\ell) \le  2n(p-n)$ for all $\ell \in [k]$. Then 
\[p^2=e(G)=e(G_1)+\cdots+e(G_k) \le 2kn(p-n),\]
which implies  that $p\le  \left\lfloor\left(1+\sqrt{1-\frac{2}{k}}\right)kn\right\rfloor$, 
contrary to   the choice of $p$. This proves that $G  \rightarrow (S(n,n);k)$, as desired. 
\end{proof}

We are also able to provide a general upper bound for $r_{bip}(S(n,m);k)$ when $m\le n<2m$.\medskip

\begin{cor}\label{c:upper2} For all integers $k \ge 3$ and $n\ge m\ge1$, we have 
\[kn+1\le r_{bip}(S(n,m);k) \le max\left\{ kn + 1, \left \lfloor\left(1+\sqrt{1-\frac{2}{k}}\right)km\right\rfloor + 1 \right \}.\]
\end{cor}

\begin{proof} Let $p := \left \lfloor\left(1+\sqrt{1-\frac{2}{k}}\right)km\right\rfloor + 1$ and $q := kn+1$.  It suffices to show that either $K_{p,p}   \rightarrow (S(n,m);k)$ or $K_{q,q}  \rightarrow (S(n,m) ; k)$.
    Suppose neither is true.  Let $G:=K_{r,r} $, where $r:=\max\{p,q\}$.    Let $\tau: E(G)\rightarrow [k]$ be a  critical coloring of $G$ with color classes $E_1, \ldots, E_k$.   For each $\ell\in[k]$,  let $G_\ell$ be the spanning subgraph of $G$ with edge set $E_\ell$.  Note that  $G_\ell$ is $S(n,m)$-free for each $\ell\in[k]$, and   $r\ge q=kn+1\ge 3n+1$.     By \cref{l:bipartite},   $e(G_\ell) \le  \max\{nr, 2m(r-m)\}$ for all $\ell \in [k]$. Thus, 
\[r^2=e(G)=e(G_1)+\cdots+e(G_k) \le \begin{cases}2km(r-m)\,\, \text{ if }\,\, nr\le 2m(r-m),\\
knr \,\,\text{ if }\,\, nr\ge2m(r-m).
\end{cases}\]
It follows   that $r\le p-1$ if $nr\le  2m(r-m)$,  and $r\le q-1$ if $nr\ge 2m(r-m)$, 
contrary to  the choice of $r$.  
\end{proof}

From the proof of \cref{t:main} when $k=2$, one can also   deduce the following exactly value of  for 
 $r_{bip}(S(n_1, m_1), S(n_2, m_2), \ldots, S(n_k, m_k))$. We omit the  proof here.

\begin{cor}
 For all integers $n_1\ge m_1\ge1$ and $n_2\ge m_2\ge1$, we have
 \[r_{bip}(S(n_1, m_1), S(n_2, m_2))=  n_1 + n_2  +1.\]
\end{cor}

\begin{cor}
  For all integers  $k\ge3$ and $n_1\ge 2m_1, \ldots, n_k\ge  2m_k$, if  $n_1+n_2+\cdots+n_k \ge 3n_i$ for each $i \in [k]$, 
then
 \[r_{bip}(S(n_1, m_1), S(n_2, m_2), \ldots, S(n_k, m_k))=  n_1 + n_2 + \ldots + n_k +1.\]
\end{cor}

The paper is organized as follows. We first prove  a Tur\'an-type theorem on bipartite graphs that are  $S(n,m)$-free in Section~\ref{Turan}, and  then  \cref{t:main} in Section~\ref{main}.

 \section{The Tur\'an argument}\label{Turan}
Following the ideas of  Irving~\cite{Irving} and  Conlon~\cite{Conlon},   we study  a Tur\'an-type problem on bipartite graphs: how many edges can a bipartite graph   have if  it is $S(n,m)$-free?   \cref{l:bipartite}, which  plays a key role in the proof of \cref{t:main},  considers the case when  the host bipartite graph is a spanning subgraph of $K_{p,p}$.  The bound in  \cref{l:bipartite} is tight.

  \begin{thm}\label{l:bipartite}
  Let $H$ be a spanning subgraph of $K_{p,\, p}$ such that $H$ is $S(n,m)$-free.  If   $p\ge 3n+1$ and $n\ge m$, then $e(H)\le \max\{np, 2m(p-m)\}$. 
\end{thm}
\begin{proof} Let $G:=K_{p,\,p}$ with bipartition $X, Y$ 
and let $H$ be a spanning subgraph of $G$ such that $H$ is $S(n,m)$-free. 
Define
\[X_1:=\{x\in X\mid d_H(x)\ge n+1\} \,\,\text{ and }  \,\,Y_1:=\{y\in Y\mid d_H(y)\ge n+1\}.\]
 We choose $H$ such  that 
\begin{enumerate}[(a)]
\item $e(H)$ is maximum among all  spanning subgraphs of $G$ that are $S(n,m)$-free; and
\item  subject to (a), $|X_1|+|Y_1|$ is minimum.
\end{enumerate}  
 Let $X_2:=N_H(Y_1)$, $Y_2:=N_H(X_1)$,  $X_3:=X\less(X_1\cup X_2)$ and $Y_3:=Y\less(Y_1\cup Y_2)$. 
Since $H$ is $S(n,m)$-free, we see that  $X_1$ is anti-complete to $Y_1$ in $H$,   $d_H(v)\le m$ for each  vertex $v \in  X_2 \cup Y_2$, and  $d_H(u)\le n$ for each  vertex $u \in X_3$ $\cup$ $Y_3$.  Assume  first $|Y_1|=0$. Then $|X_2|=0$ and  $|Y_2|+|Y_3|=p$.  Thus   $e(H)=e_H(Y_2, X)+e_H(Y_3, X) \le m|Y_2|+n|Y_3|\le n(|Y_2|+|Y_3|)=np $, as desired. We may assume that $|Y_1|\ne 0$. Similarly, $|X_1|\ne 0$.  Then $|X_2|\ge n+1$ and $|Y_2|\ge n+1$. We next claim that $|X_1|\le m$ and $|Y_1|\le m$. \medskip

Suppose   $|X_1|\ge m+1$  or $|Y_1|\ge m+1$, say  the latter.    Let $y\in Y_1$ with $d_H(y)=q\ge n+1$. Let $x_1, \ldots, x_{q-n}\in  X_2$ be distinct neighbors of $y$. For each $i\in[q-n]$, since $d_H(x_i)\le m<|Y_1|$, there must exist a vertex $y_i\in Y_1$ with $y_i\ne y$ such that $x_iy_i\notin E(H)$. Let $H^*$ be obtained from $H$ by first deleting edges $yx_1, \ldots, yx_{q-n}$, and then adding edges $x_1y_1, \ldots, x_{q-n}y_{q-n}$. It is simple to check that $H^*$ is $S(n,m)$-free, $e(H^*)=e(H)$ and $d_{H^*}(y)=n$. But then  
\[|\{v\in V(H^*)\mid d_{H^*}(v)\ge n+1\}|=|X_1|+|Y_1\less y|<|X_1|+|Y_1|,\] contrary to the choice of $H$.  This proves that 
$|X_1|\le m$ and $|Y_1|\le m$, as claimed. \medskip
 
For the remainder of the proof, we assume that $|X_2|\le |Y_2|$. 
  It follows that   
\[e(H)=e_H(Y_1, X_2)+e_H(Y_2, X)+e_H(Y_3, X) \le |X_2||Y_1|+m|Y_2|+n|Y_3|\le (|Y_1|+m)|Y_2|+n|Y_3|. \tag{$*$}\]

We first consider the case   $|Y_2|\ge p-m$. Let $|Y_2|:=p-m+r$ for some integer $r$ satisfying $0\le r\le m$. Then $|Y_1|=p-|Y_2|-|Y_3|=p-(p-m+r)-|Y_3|=m-r-|Y_3|$. By ($*$), 
\begin{align*}
e(H)&\le  (|Y_1|+m)|Y_2|+n|Y_3|\\
&= (m-r-|Y_3|+m)(p-m+r)+n|Y_3|\\
&=2m(p-m)-(p+r-3m)r -(p+r-n-m)|Y_3|\\
&\le 2m(p-m), 
\end{align*}
as desired, because $p\ge 3n+1$ and $n\ge m$. \medskip

It remains to consider the case $n+1\le |X_2|\le |Y_2|<p-m$.  By ($*$), 
\begin{align*}
e(H)&\le  (|Y_1|+m)|Y_2|+n|Y_3|\\
&=(|Y_1|+m)|Y_2|+  n(p-|Y_1|-|Y_2|)\\
&= (|Y_2|-n)|Y_1|+np+(m-n)|Y_2|\\
&\le (|Y_2|-n)m+np+(m-n)|Y_2|\\
&=np-nm+(2m-n)|Y_2|\\
&\le \begin{cases}
 np  &\text{ if } n\ge 2m\\
 np-nm+(2m-n)(p-m)=2m(p-m)  &\text{ if } n< 2m,
 \end{cases}
\end{align*}
as desired. This completes the proof of \cref{l:bipartite}.
\end{proof}

\section{Proof of \cref{t:main}}\label{main}

We begin with a lemma on the lower bound for the bipartite Ramsey number of double stars. 

\begin{lem}\label{l:lowbd}   For all   $k\ge1$ and $n\ge m\ge1$, we have 
 \[r_{bip}(S(n,m);k)\ge  kn+1.\]
 \end{lem}
\begin{proof} Let $G:=K_{kn, kn}$.  It suffices to show that $G  \nrightarrow (S(n,m);k)$. Note that  $G$ has a proper $kn$-edge-coloring, and so $G$ has $kn$ pairwise disjoint perfect matchings, say    $M_1, \ldots, M_{kn}$. Let $\tau: E(G)\rightarrow [k]$ be obtaining by coloring all edges in $M_{\ell n+1}\cup \cdots \cup M_{\ell n+n}$ by color $\ell+1$ for each $\ell \in\{0, \ldots, k-1\}$. It follows that $G$ contains no monochromatic $S(n,m)$ under $\tau$, as $S(n,m)$ has maximum degree $n+1$.
\end{proof}

We are now ready to prove \cref{t:main},  which we restate here for convenience.\doublestar*

\begin{proof} Let $k, n,m$ be as given in the statement.    By \cref{l:lowbd}, it suffices to show that $G  \rightarrow (S(n,m);k)$, where  $G:=K_{kn+1, kn+1}$.    
  Suppose $G  \nrightarrow (S(n,m);k)$. 
Let  $X, Y$ be the bipartition of $G$.   Let $\tau: E(G)\rightarrow [k]$ be a critical $k$-coloring $G$ with color classes $E_1, \ldots, E_k$.   For each $\ell\in[k]$,  let $G_\ell$ be the spanning subgraph of $G$ with edge set $E_\ell$.   Then $G_\ell$ is $S(n,m)$-free. Suppose $k\ge3$ and $n\ge 2m$. By \cref{l:bipartite}, $e(G_\ell)\le n(kn+1)$ for all $\ell\in[k]$. Then 
\[(kn+1)^2=e(G)=e(G_1)+\cdots+e(G_k)\le kn(kn+1),\]
which is impossible. Thus $k=2$ and so $|X|=|Y|=2n+1$.  We may further assume that color $1$ is blue and color $2$ is red. We denote $G_1$ and $G_2$ by $G_b$ and $G_r$, respectively. 
Let 
\[A:=\{x\in X\mid d_{G_b}(x)\ge n+1\} \text{ and } B:=\{y\in Y\mid yx\in E(G_b) \text{ for some } x\in A\}.\]
 Since $G_b$ is $S(n,m)$-free, we see that  $d_{G_b}(v)\le m$ for each  $v\in B$. Note that  $|X|=|Y|=2n+1$ and $n\ge m$.  It follows that  $d_{G_r}(x)\ge n+1$ for each $x\in X\less A$,  and $d_{G_r}(y)\ge n+1$ for each $y\in B$. Thus  $B$ is anti-complete to $X\less A$ in $G_r$ because $G_r$ is $S(n,m)$-free, and so $N_{G_r}(v)\subseteq A$ for each $v\in B$. 
 But then 
 \[(n+1)|A|\le e_{G_b}(A, B)\le m|B|\,\, \text{ and } \,\,(n+1)|B|\le  e_{G_r}(B, A)\le m|A|,\]
 which is impossible because $n\ge m$. 
\end{proof}

\end{document}